\documentclass[11pt]{article}
\usepackage{fullpage} 
\usepackage{blkarray}
\usepackage{ucs}
\usepackage{amssymb}
\usepackage{amsthm}
\usepackage{amsmath}
\usepackage{latexsym}
\usepackage[utf8]{inputenc}
\usepackage[english]{babel}
\usepackage{graphicx}
\usepackage{wrapfig}
\usepackage{caption}
\usepackage{subcaption}
\usepackage{txfonts}
\usepackage{lmodern}
\usepackage{mathrsfs}
\usepackage{algorithm}
\usepackage{dsfont}
\usepackage{enumerate}
\usepackage[hidelinks]{hyperref}
\usepackage{multicol}
\usepackage{csquotes}
\usepackage{stmaryrd}
\usepackage{tikz}
\usepackage{comment}
\usepackage{enumitem}

\usepackage{cite}

\newcommand{\badrounding}[1]{}

\tikzset{vtx/.style={inner sep=1.7pt, outer sep=0pt, circle, fill,draw}}

\tikzset{vtxB/.style={inner sep=1.7pt, outer sep=0pt, circle, fill=white,draw}}

\tikzset{vtx/.style={inner sep=1.7pt, outer sep=0pt, circle, fill}} 
\tikzset{unlabeled_vertex/.style={inner sep=1.7pt, outer sep=0pt, circle, fill}} 
\tikzset{labeled_vertex/.style={inner sep=2.2pt, outer sep=0pt, rectangle, fill=yellow, draw=black}} 
\tikzset{edge_color0/.style={color=black,line width=1.2pt}} 
\tikzset{edge_color1/.style={color=red,  line width=1.2pt,opacity=0}} 
\tikzset{edge_color2/.style={color=blue, line width=1.2pt,opacity=1}} 
\tikzset{edge_color2d/.style={color=blue, line width=1.2pt,opacity=1,dashed}} 
\tikzset{edge_color3/.style={color=blue, line width=1.2pt,opacity=1,dashed}} 
\tikzset{edge_color4/.style={color=red,  line width=1.2pt,dotted}} 
\tikzset{edge_color5/.style={color=blue, line width=1.2pt,dotted}} 
\tikzset{edge_color6/.style={color=green, line width=1.2pt,dotted}} 
\tikzset{vertex_color1/.style={inner sep=1.7pt, outer sep=0pt, draw, circle, fill=red}} 
\tikzset{vertex_color2/.style={inner sep=1.7pt, outer sep=0pt, draw, circle, fill=blue}} 
\tikzset{vertex_color3/.style={inner sep=1.7pt, outer sep=0pt, draw, circle, fill=green}}

\tikzset{vertex_u/.style={inner sep=1.7pt, outer sep=0pt, circle, fill}} 
\tikzset{vertex_l/.style={inner sep=2.2pt, outer sep=0pt, rectangle,fill=yellow, draw=black}}

\newtheorem{theo}{Theorem}

\newtheorem{lemma}[theo]{Lemma}

\newtheorem{corl}[theo]{Corollary}
\newtheorem{conj}[theo]{Conjecture}

\theoremstyle{definition}

\numberwithin{theo}{section}

\author{%
J\'ozsef Balogh\footnote{Department of Mathematics, University of Illinois at Urbana-Champaign, Urbana, Illinois 61801, USA. E-mail: \texttt{jobal@illinois.edu}. Research is partially supported by NSF Grants DMS-1764123, RTG DMS 1937241 and FRG DMS-2152488, Arnold O. Beckman Research Award (UIUC Campus Research Board RB 22000), the Langan Scholar Fund (UIUC), and the Simons Fellowship.} 
\and Felix Christian Clemen \footnote {Department of Mathematics, University of Illinois at Urbana-Champaign, Urbana, Illinois 61801, USA, E-mail: \texttt{fclemen2@illinois.edu}.}
 \and Bernard Lidick\'{y} \footnote {Iowa State University, Department of Mathematics, Iowa State University, Ames, IA., E-mail: \texttt{lidicky@} \texttt{iastate.edu}. Research of this author is partially supported by NSF grants DMS-1855653 and FRG DMS-2152498 and Scott Hanna fellowship.}
\and
Sergey Norin\footnote{Department of Mathematics and Statistics, McGill University, Montreal, QC, Canada.
\texttt{sergey.norin@mcgill.ca}. Research of this author was supported by an NSERC Discovery Grant.}
\and 
Jan Volec\footnote{
Department of Mathematics, Faculty of Nuclear Sciences and Physical Engineering, Czech Technical University in Prague, Prague, Czech Republic.
\texttt{jan@ucw.cz}. Research of this author is partially supported by Czech Science Foundation GA\v{C}R, grant 23-06815M.}
}

\newcommand{\vc}[2][0.34]{\ensuremath{\vcenter{\hbox{\includegraphics[scale=#1,page=#2]{flags.pdf}}}}}
\newcommand{\cF}{\mathcal{F}}

\def\flPath{{\vc{38}}}

\def\flChry{{\vc{15}}}

\def\flRcochr{{\vc{17}}}
\def\flRedge{{\vc{18}}}
\def\flRpath{{\vc{19}}}
\def\flRchry{{\vc{20}}}
\def\flRNEall{{\vc{23}}} \def\flRNEone{{\vc{22}}} \def\flRNEtwo{{\vc{21}}}
\def\flREnone{{\vc{39}}} \def\flREone{{\vc{41}}} \def\flREtwo{{\vc{40}}}

\def\flRCOCHone{{\vc{31}}} \def\flRCOCHtwo{{\vc{32}}}
\def\flRCOCHthree{{\vc{33}}}
\def\flRCOCHfour{{\vc{34}}} \def\flRCOCHfive{{\vc{35}}}

\def\flREMPToneA{{\vc{24}}} \def\flREMPToneB{{\vc{25}}} \def\flREMPToneC{{\vc{26}}}
\def\flREMPTtwoA{{\vc{27}}} \def\flREMPTtwoB{{\vc{28}}} \def\flREMPTtwoC{{\vc{29}}}
\def\flREMPTthree{{\vc{30}}}

\DeclareMathOperator{\Aut}{Aut}

\newcommand{\eps}{\varepsilon}
\newcommand{\sh}{\operatorname{hom}_s}
\def\brandtconst{{0.15442}}
\def\nikifconst{{0.15467}}
\def\janconst{{\frac{15}{94}}}

\title{The Spectrum of  Triangle-free Graphs}
\begin{document}
\maketitle 
\begin{abstract} Denote by $q_n(G)$ the smallest eigenvalue of the signless Laplacian matrix of an $n$-vertex graph $G$.
Brandt conjectured in 1997 that for regular triangle-free graphs $q_n(G) \leq \frac{4n}{25}$.
We prove a stronger result: If $G$ is a triangle-free graph then $q_n(G) \leq \frac{15n}{94}< \frac{4n}{25}$.
Brandt's conjecture is a subproblem of two famous conjectures of Erd\H{o}s:

    (1) Sparse-Half-Conjecture: Every  $n$-vertex triangle-free graph has a subset of vertices of size $\left\lfloor\frac{n}{2}\right\rfloor$ spanning at most $n^2/50$ edges. 
    
    (2) Every $n$-vertex triangle-free graph can be made bipartite by removing at most $n^2/25$ edges. 
    
In our proof we use linear algebraic methods to upper bound $q_n(G)$ by the ratio between the number of induced paths with 3 and 4 vertices. We give an upper bound on this ratio via the method of flag algebras. 
    \end{abstract}
    \section{Introduction}
We prove a result on eigenvalues of triangle-free graphs which is motivated by the following two famous conjectures of Erd\H{o}s.
\begin{conj}[Erd\H{o}s' \emph{Sparse Half Conjecture}~\cite{MR0409246,MR1439273}]
\label{sparsehalf}
Every triangle-free graph on $n$ vertices has a subset of vertices of size $\left\lfloor\frac{n}{2}\right\rfloor$ vertices spanning at most $\left\lfloor \frac{n^2}{50}\right\rfloor$ edges.
\end{conj}
Erd\H{o}s offered a \$250 reward for proving this conjecture. There has been progress on this conjecture in various directions~\cite{MR1273598,MR2232396,sparsehalfraz,MR3383248,SchachtErdos}. Most recently, Razborov~\cite{sparsehalfraz} proved that every triangle-free graph on $n$ vertices has an induced subgraph on $\lfloor n/2\rfloor$ vertices with at most $(27/1024)n^2$ edges. 

For a graph $G$, denote by $D_2(G)$ the minimum number of edges which have to be removed to make $G$ bipartite.
\begin{conj}[Erd\H{o}s~\cite{MR0409246}]
\label{makingbipartite}
Let $G$ be a triangle-free graph on $n$ vertices. Then $D_2(G)\leq n^2/25$.
\end{conj}
There also has been work on this conjecture~\cite{Howtotriangle,Howmany,flagbipartite,Alo96,She92}, most recently, Balogh, Clemen and Lidick\'y~\cite{flagbipartite} proved $D_2(G)\leq n^2/23.5$.

Brandt~\cite{MR1606776} found a surprising connection between these two conjectures and the eigenvalues of triangle-free graphs. Denote by $\lambda_n(G)\leq \ldots \leq \lambda_1(G)$ the eigenvalues of the adjacency matrix of an $n$-vertex graph $G$. Brandt~\cite{MR1606776} proved that 
\begin{align}
\label{D2GBrandt}
    D_2(G)\geq \frac{\lambda_1(G)+\lambda_n(G)}{4}\cdot n
\end{align}
for regular graphs and conjectured the following.
\begin{conj}[Brandt~\cite{MR1606776}]
\label{brandtconj}
Let $G$ be a triangle-free regular $n$-vertex graph. Then
\[\lambda_1(G)+\lambda_n(G)\leq \frac{4}{25} \cdot n.\]
\end{conj}

Towards this conjecture, Brandt~\cite{MR1606776} proved a bound $\lambda_1(G)+\lambda_n(G)\leq (3-2\sqrt{2})n\approx 0.1715n$ for regular triangle-free graphs, which was very recently shown to hold also in the non-regular setting by Csikvári~\cite{arxivCsi22}.
Brandt also noted that $\lambda_1(G_{HS})+\lambda_n(G_{HS})=0.14n$ for the so-called Higman-Sims graph $G_{HS}$, which is the unique strongly regular graph with parameters $(n,d,t,k) =(100,22,0,6)$.
Recall that an $(n,d,t,k)$-\emph{strongly regular graph} is an $n$-vertex $d$-regular graph, where the number of common neighbors of every pair of adjacent vertices is $t$ and the number of common neighbors of a non-adjacent pair of vertices is $k$.

The value $4/25$ is motivated by the fact that if either of Conjectures~\ref{sparsehalf} or \ref{makingbipartite} were true, it would imply Conjecture~\ref{brandtconj}. As observed by Brandt~\cite{MR1606776}, Conjecture~\ref{sparsehalf} implies Conjecture~\ref{brandtconj} by applying the following version of the Expander Mixing Lemma for a set $S\subset V(G)$ of size $n/2$ with $e(S)\leq n^2/50$. 
\begin{lemma}[Bussemaker-Cvetkovi\'c-Seidel~\cite{MR519264}, Alon-Chung~\cite{MR975519}]
Let $G$ be an $n$-vertex $d$-regular graph. Then, for every $S\subseteq V(G)$, we have
\begin{align*}
    e(S)\geq |S|\cdot\frac{|S|d+(n-|S|)\lambda_n(G)}{2n}.
\end{align*}
\end{lemma}

Given a graph $G$, denote by $Q=A+D$ the \emph{signless Laplacian matrix} of $G$, where $D$ is the diagonal matrix of the degrees of $G$ and $A$ is the adjacency matrix of $G$. Let $q_n(G)\leq\ldots\leq q_1(G)$ be the eigenvalues of $Q$. By considering the signless Laplacian matrix, De Lima, Nikiforov and Olivera~\cite{MR3477106} extended \eqref{D2GBrandt} beyond regular graphs as follows.

\begin{theo}[De Lima, Nikiforov and Olivera\cite{MR3477106}]\label{QD2}
	For every $n$-vertex graph $G$ we have $$D_2(G) \geq \frac{q_n(G)}{4}\cdot n.$$
\end{theo}

By Theorem~\ref{QD2}, if Conjecture~\ref{makingbipartite} holds then  $q_n(G) \leq \frac{4n}{25}$ for every triangle-free $n$-vertex graph $G$. Motivated by this observation De Lima, Nikiforov and Olivera~\cite{MR3477106} proposed investigating upper bounds on $q_n(G)$, and proved $q_n(G) \leq \frac{2n}{9}$ for $n$-vertex triangle-free graphs $G$.  
Our main result is an improvement of this bound, which solves Conjecture~\ref{brandtconj}.
\begin{theo}\label{Nikifbound}
	If $G$ is a triangle-free $n$-vertex graph, then $$q_n(G) \leq \janconst \cdot n < 0.1596n.$$
\end{theo}

Note that, if $G$ is $d$-regular, then $\lambda_1(G) = d$ and $q_n(G) = \lambda_n(G) + d=\lambda_n(G)+\lambda_1(G)$. Thus  Theorem~\ref{Nikifbound} implies that $\lambda_1(G)+\lambda_n(G) < 0.1596 n < \frac{4n}{25}$ for every regular triangle-free $n$-vertex graph $G$, confirming Conjecture~\ref{brandtconj} in strong form.


It remains open to determine a sharp upper bound for  $q_n(G)/n$ for triangle-free $n$-vertex graph $G$.
While we only prove Theorem~\ref{Nikifbound} with the constant $\frac{15}{94} \approx 0.1596$, a larger flag algebra computation yields $q_n(G) < \nikifconst n$.
Also, one can additionally assume that $G$ is regular and use flag algebras to show a slightly stronger bound $q_n(G)=\lambda_1(G)+\lambda_n(G) < \brandtconst n$.
As we believe neither of these two bounds are sharp (see Section~\ref{sec:outro}), we omit presenting their proofs.

\section{Proof of Theorem~\ref{Nikifbound}}
Our proof is based on bounding the ratio between the number of induced paths with $3$ and $4$ vertices in triangle-free graphs.
On one hand, we upper bound $q_n(G)$ in terms of this ratio in Lemma~\ref{l:rayleigh} and Corollary~\ref{c:qn}. On the other hand, Lemma~\ref{l:flag}, which is proved using flag algebras, gives a sufficiently good bound on the ratio.

For an edge $e=xy$ of a graph $G$,
let $m_{xy}$ be the number of edges $uv \in E(G)$  such that $ux,vy \in E(G)$.
For a vertex $x \in V(G)$, let $w_x$ to be the number of walks of length two starting in $x$, i.e. $w_x$ is the number of edges $uv \in E(G)$ such that $xu \in E(G)$. 

\begin{lemma}\label{l:rayleigh}
	If $G$ is an $n$-vertex triangle-free graph and $xy \in E(G)$, then
\begin{equation}\label{e:xy}
	\left( \deg(x) + \deg(y) \right) \cdot q_n(G) \leq  w_x + w_y-2m_{xy}\,.
\end{equation}
\end{lemma}

\begin{proof}
	Define a vector $z = (z_v)_{v \in V(G)} \in \mathbb{R}^{V(G)}$  by
	\begin{equation*}
	z_v = \begin{cases}
	+1, & \text{if } xv \in E(G), \\
	-1, & \text{if } yv \in E(G), \\
	0, & \text{otherwise.} 
	\end{cases}\end{equation*}
	The vector $z$ is well-defined since $G$ is triangle-free. Also note that $\| z \|^2 = \deg(x) + \deg(y)$. Let $Q$ be the signless Laplacian matrix of $G$.
	We have \begin{align*}z^T Q z &= \sum_{u,v \in V(G)}  Q_{uv}z_uz_v = \sum_{u \in V(G)}(z_u)^2 \deg(u) +  2\cdot\sum_{
		uv \in E(G)}z_uz_v \\
		&=  w_x+w_y+  2\cdot \sum_{
		uv \in E(G)}z_uz_v=  w_x + w_y-2m_{xy}, \end{align*} 
		where in the last equality we used that $G$ is triangle-free.
	Since $Q$ is symmetric, $q_n(G)$ is upper bounded by the Rayleigh-Ritz quotient of $z$, i.e.
	\[q_n(G) \leq \frac{z^T Q z}{\|z\|^2} = \frac{w_x + w_y - 2m_{xy}}{\deg(x)+\deg(y)}\,,\]
	as desired.
\end{proof}

A map $\varphi:V(H) \to V(G)$ is \emph{a strong homomorphism} from a graph $H$ to a graph $G$ if for every pair of vertices $u,v \in V(H)$ we have $uv \in E(H)$ if and only if  $\varphi(u)\varphi(v) \in E(G)$.
Let $\sh(H,G)$ denote the number of strong homomorphisms from $H$ to $G$.
Let $P_k$ denote the $k$-vertex path. Summing the bound from Lemma~\ref{l:rayleigh} over all the edges of $G$ yields the following.


\begin{corl}\label{c:qn}
	If $G$ is an $n$-vertex triangle-free graph, then \begin{equation}\label{e:sh}\sh(P_3,G)\cdot q_n(G) \le  \sh(P_4,G)\,.\end{equation}
\end{corl}

\begin{proof} First, note that \begin{equation}\label{e:p2}\sum_{xy \in E(G)}(\deg(x)+\deg(y)) = \sum_{x \in V(G)} \deg^2(x) = \sh(P_3,G),\end{equation}
where in the last equality we used that $G$ is triangle-free. Meanwhile, $\sum_{xy \in E(G)}(w_x+w_y)$ is equal to the number of walks of length three in $G$, i.e. the number of maps $\phi: \{1,2,3,4\} \to V(G)$ such that $\{\phi(1)\phi(2),\phi(2)\phi(3),\phi(3)\phi(4)\} \subset E(G)$. Similarly, the expression $2\sum_{xy \in E(G)}m_{xy}$ is equal to the number of maps $\phi: \{1,2,3,4\} \to V(G)$ such that $\{\phi(1)\phi(2), \phi(2)\phi(3), \phi(3)\phi(4),\phi(4)\phi(1)\} \subset E(G)$. It follows that $\sum_{xy \in E(G)}(w_x+w_y -2m_{xy})$ counts the maps $\psi: \{1,2,3,4\} \to V(G)$ such that $\{\psi(1)\psi(2),\psi(2)\psi(3),\psi(3)\psi(4)\} \subset E(G)$ and $\psi(4)\psi(1) \not \in E(G)$, i.e., 
 \begin{equation}\label{e:p3}\sum_{xy \in E(G)}(w_x+w_y -2m_{xy}) = \sh(P_4,G).\end{equation}
 Summing \eqref{e:xy} over all $xy \in E(G)$ and using  \eqref{e:p2} and \eqref{e:p3}, we obtain \eqref{e:sh}.
\end{proof}

Theorem \ref{Nikifbound} is an immediate consequence of the above corollary and the following lemma which is proved using standard, albeit computer-assisted flag-algebra calculation.

\begin{lemma}\label{l:flag}
If $G$ is an $n$-vertex triangle-free graph, then
\begin{equation}\label{eq:hom}
\sh(P_4,G) \le \frac{15n}{94} \cdot \sh(P_3,G).
\end{equation}
\end{lemma}

\begin{proof}
Suppose the lemma is false, and let $G$ be an $n$-vertex triangle-free graph such that \begin{equation}\label{eq:hom2}\sh(P_4,G) =\frac{15n}{94} \cdot \sh(P_3,G) + \eps n^4\,,\end{equation} for some $\eps  > 0$.
Let $G^{(b)}$ be the $b$-blowup of $G$, obtained by replacing every vertex of $G$ by $b$ pairwise non-adjacent vertices.  Then $\sh(P_k,G^{(b)}) = \sh(P_k,G) \cdot b^{k}$ for $k=3,4$.
In particular, for every $b \in \mathbb{N}$, the graph $G^{(b)}$ satisfies the analogue of \eqref{eq:hom2} as well.

Let us now reformulate \eqref{eq:hom2} in the flag algebra language~\cite{flagsRaz}.
Given a graph $H$, let $p\left(\flChry,H\right)$ be the probability that a $3$-vertex subset of $V(H)$ chosen uniformly at random induces exactly two edges.
Analogously, let $p\left(\flPath,H\right)$ be the probability that a randomly chosen $4$-vertex subset induces a path of length $3$.

For every fixed $\ell$-vertex graph $F$ and a $k$-vertex graph $H$, only $O(k^{\ell - 1})$ maps $V(F) \to V(H)$ are non-injective. Therefore, $\sh(F,H) = |\Aut(F)| \cdot p(F,H) \cdot {k \choose \ell} + O(k^{\ell -1})$, so
in particular every $k$-vertex triangle-free graph $H$ satisfies
\[\sh(P_4,H) = \frac{k^4}{12} \cdot p\left(\flPath,H\right) + O\left(k^3\right)
\quad\hbox{and}\quad
\sh(P_3,H) = \frac{k^3}3 \cdot p\left(\flChry,H\right) + O\left(k^2\right)\,.\]
Recall that $G^{(b)}$ satisfies~\eqref{eq:hom2}. Multiplying it by $564 / (bn)^4$ and rearranging yields that
\begin{equation*}
\lim_{b \to \infty} \; 30 \cdot p\left(\flChry,G^{(b)}\right) - 47 \cdot p\left(\flPath,G^{(b)}\right) = -564\eps^4.
\label{eq:counterex}
\end{equation*}
In order to derive a contradiction, we present a flag algebra computation proving that an inequality
\begin{equation}\label{eq:flag}
30 \cdot \flChry - {47}\cdot \flPath \ge 0
\end{equation}
asymptotically holds in the theory of triangle-free graphs. To see that, consider the following $6$~flag-algebra expressions, which are all non-negative:
\begin{align*}
\mbox{1)\quad\quad} &{\left(13\cdot\left(\flREMPToneA+\flREMPToneB+\flREMPToneC\right) - 52\cdot\left(\flREMPTtwoA + \flREMPTtwoB +\flREMPTtwoC\right)+84\cdot\flREMPTthree  \right)^2}
\\
\mbox{2)\quad\quad} &{\left(31\cdot\left(\flRCOCHone+\flRCOCHtwo\right) - 63\cdot\left(\flRCOCHfour + \flRCOCHfive\right)+3\cdot\flRCOCHthree  \right)^2}
\\
\mbox{3)\quad\quad} &{\left(94\cdot\flRcochr -55\cdot\flRedge -14\cdot\flRpath +58\cdot\flRchry  \right)^2}
\\
\mbox{4)\quad\quad} &{\flRNEone \times \left(2\cdot\flRNEone + 10 \cdot \flRNEtwo - 24\cdot \flRNEall   \right)^2}
\\
\mbox{5)\quad\quad} &{\flREone \times \left(14\cdot \flREnone + 19\cdot\flREone - 44 \cdot \flREtwo \right)^2}
\\
\mbox{6)\quad\quad} &{\flREone \times \left(9\cdot \flREnone - 14\cdot\flREone - 3 \cdot \flREtwo \right)^2}
\end{align*}

\begin{figure}
\includegraphics[scale=0.85,page=1]{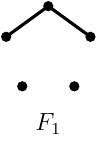}
\hfill\includegraphics[scale=0.85,page=2]{flags.pdf}
\hfill\includegraphics[scale=0.85,page=3]{flags.pdf}
\hfill\includegraphics[scale=0.85,page=4]{flags.pdf}
\hfill\includegraphics[scale=0.85,page=5]{flags.pdf}
\hfill\includegraphics[scale=0.85,page=6]{flags.pdf}

\vskip 4mm

\includegraphics[scale=0.85,page=7]{flags.pdf}
\hfill\includegraphics[scale=0.85,page=8]{flags.pdf}
\hfill\includegraphics[scale=0.85,page=9]{flags.pdf}
\hfill\includegraphics[scale=0.85,page=10]{flags.pdf}
\hfill\includegraphics[scale=0.85,page=11]{flags.pdf}
\hfill\includegraphics[scale=0.85,page=12]{flags.pdf}
\caption{The set $\cF$ of $5$-vertex triangle-free graphs with at least $2$ edges.} \label{fig:5flags}
\end{figure}

Let $\cF$ be the set of all the $5$-vertex triangle-free graphs with at least $2$ edges.
A case analysis yields $|\cF| = 12$; see Figure~\ref{fig:5flags}.
Now observe that averaging over all choices of the labelled vertices in each of the $6$ expressions yields a linear combination of subgraph densities, where every term has $5$ vertices and at least $2$ edges.
Thus a flag algebra argument yields that the average of the $i$-th expression is equal to the $i$-th coordinate of $M \cdot \left(v_{\cF}\right)^T$, where $v_{\cF} = \left(F_1,\dots,F_{12}\right)$ and 
\[
\tiny
M=\frac1{30} \times
\begin{blockarray}{cccccccccccc}
  \begin{block}{(cccccccccccc)}
    507&2028&0&-4056&-3549&0&1248&8112&16224&-13104&0&21168&\\
    0&0&2883&381&961&0&-3906&-4098&3844&63&19845&0&\\
    12100&-23688&-19140&-23620&12172&20184&-37248&17486&47664&2956&86730&-7392&\\
    0&0&6&140&0&0&-48&-100&0&358&-1200&0&\\
    196&0&798&196&-420&2166&762&-1036&-2464&-702&-3080&792&\\
    81&0&-378&81&54&1176&-165&27&-108&-87&-135&279&\\
  \end{block}
\end{blockarray}
\,.
\]
On the other hand, another flag algebra argument yields that the left-hand side of~\eqref{eq:flag} is equal to
{\[
 3 \cdot F_1 +9 \cdot F_3 +3 \cdot F_4 -\frac{17}{5} \cdot F_5 +18 \cdot F_6 -\frac{34}{5} \cdot F_7 -\frac{49}{5} \cdot F_8 +12 \cdot F_9 -\frac{4}{5} \cdot F_{10} -32 \cdot F_{11} +27 \cdot F_{12} .
\]}
A tedious yet straightforward calculation reveals the following coordinate-wise inequality
{\[
\left(\frac1{33}, \frac{12}{209}, \frac{3}{1147}, \frac{231}{163}, \frac{17}{84},\frac{12}{293}   \right) \cdot M
<
\left( 3, 0, 9, 3, -\frac{17}{5},18, -\frac{34}{5}, -\frac{49}{5}, 12, -\frac{4}{5}, -32 , 27  \right),
\] }
which in turn shows that \eqref{eq:flag} asymptotically holds in the theory of triangle-free graphs.
\end{proof}

The flag algebra calculations used in the proof of Lemma~\ref{l:flag} can be independently verified by a SAGE script,
which is available as an ancillary file of the arXiv version of this manuscript.

\section{Concluding remarks}\label{sec:outro}
As we have already mentioned in the introduction, a significantly larger flag algebra computation than the one used in our proof yields that $q_n(G) < \nikifconst n$ for every triangle-free $n$-vertex graph.
Similarly, assuming that $G$ is regular allows us to show $\lambda_1(G)+\lambda_n(G) < \brandtconst n$.
On the other hand, our method will be able to get neither of the coefficients below $42/275 = 0.15\overline{{}27}$.

Indeed, consider the Higman-Sims graph $G_{HS}$. It is edge-transitive  so $m_{xy} = 21 \cdot 6 + 22$ for every $xy \in E(G_{HS})$, and $w_x=22^2$ for every $x \in V(G)$, where $m_{xy}$ and $w_x$ are defined as before Lemma~\ref{l:rayleigh}. 
Therefore, \[\frac{w_x +w_y - 2m_{xy}}{(\deg(x)+\deg(y)) \cdot |V(G_{HS})|} = \frac{2(22 ^2 - 21 \cdot 6 - 22)}{2 \cdot 22 \cdot 100} = \frac{42}{275},\] 
for every  $xy \in E(G_{HS})$, and so Lemma~\ref{l:rayleigh} only yields $q_n(G_{HS}) \leq \frac{42}{275} \cdot |V(G_{HS})|$.
However, we have $q_n(G_{HS}) = \lambda_1(G_{HS})+\lambda_n(G_{HS}) = 0.14 \cdot |V(G_{HS})|$. It might be that $q_n(G) \le 0.14 n$ holds for every triangle-free graph $G$ on $n$ vertices.

\section*{Acknowledgements}
We thank anonymous referees for carefully reading the manuscript.

\bibliographystyle{abbrvurl}
\bibliography{combined}

\end{document}